\documentclass[a4paper,12pt]{article}
\usepackage[latin1]{inputenc}
\usepackage{amsmath,amssymb,amsthm,cite,verbatim}
\usepackage[all]{xy}
\usepackage[bookmarks=true]{hyperref}
\usepackage{color}
\usepackage{graphicx}

\title{Functions and differentials on the non-split Cartan modular curve of level 11}
\author{Julio Fern\'andez \ and \ Josep Gonz\'alez
\footnote{The authors are partially supported  by DGICYT Grant MTM2012-34611}}

\newtheorem{prop}{Proposition}[section]

\newtheorem{lemma}{Lemma}[section]
\newtheorem{thm}{Theorem}[section]
\newtheorem{cor}{Corollary}[section]
\newtheorem{rem}{Remark}[section]

\theoremstyle{definition}

\theoremstyle{remark}

\numberwithin{equation}{section}

\newcommand{\ds}{\displaystyle}
\newcommand{\Q}{\mathbb{Q}}
\newcommand{\Qbar}{\overline{\mathbb{Q}}}

\newcommand{\Z}{\mathbb{Z}}
\newcommand{\F}{\mathbb{F}}

\newcommand{\PP}{\mathbb{P}}
\newcommand{\C}{\mathbb{C}}

\newcommand{\GL}{\operatorname{GL}}
\newcommand{\SL}{\operatorname{SL}}

\newcommand{\Jac}{\operatorname{Jac}}

\newcommand{\new}{\operatorname{new}}

\newcommand{\dv}{\operatorname{div}}

\newcommand{\cC}{{\mathcal C}}

\newcommand{\cX}{{\mathcal X}}

\begin{document}
\maketitle

\begin{abstract}
The genus $4$ modular curve $X_{ns}(11)$ attached to a non-split Cartan group of level~$11$ admits a model
defined over $\Q$. We compute generators for its function field in terms of Siegel modular functions.
We also show that its Jacobian is isomorphic over $\Q$ to the new part of the Jacobian of the classical
modular curve $X_0(121)$.
\end{abstract}

\section*{Introduction}

When making computations on the modular curve attached to a congruence subgroup~$\Gamma$ of level~$N$
containing $\Gamma_1(N)$, one counts on the $\Q$--rationality of the cusp at infinity. Moreover, several well-known
classical modular tools, such as the Dedekind $\eta$-function, are then available, together with explicit methods to find
a basis for the vector space of weight~$2$ cuspforms, that is, a basis of regular differentials on the curve.
Any such a cuspform $f(\tau)$ has period $1$ on the complex upper half-plane, hence admits a Fourier expansion
in $e^{2\pi i\tau}$ and, furthermore, the minimal field of definition for $f(\tau)$ is the number field generated by
its Fourier coefficients. All of this provides some crucial help in order to compute
generators for the function fields of the curve and its quotients.
See, for instance, \cite{gon91}, \cite{shi95}, \cite{BGGP}, \cite{Gon}.

\vspace{2truemm}

By constrast, this kind of computation becomes more complicated and laborious whenever~$\Gamma$
does not contain $\Gamma_1(N)$. This article is meant to serve as an explicit example in which $\Gamma$
is the congruence subgroup corresponding to a non-split Cartan group of level~$N\!=\!11$.
The attached modular curve $X_{ns}(11)$ has genus $4$ and is not hyperelliptic. Though it is defined over $\Q$,
its cusps are not. Even worse, for a function on the curve, its minimal field of definition does not
necessarily coincide with the number field generated by its Fourier coefficients. The interest on the modular
curves $X_{ns}(N)$ is closely related to the class number one problem and comes from the classification in \cite{Se}
of the cases in which the Galois representation attached to the $N$-torsion of an elliptic curve is not surjective.

\vspace{2truemm}

An equation for $X_{ns}(11)$ is derived in \cite{DFGS} from a certain elliptic quotient $X^+_{ns}(11)$
which was first handled in \cite{Lig}. This is done by combining a result of \cite{Baran} on
the ramification locus of the modular $j$--function together with an expression given in \cite{Hal} for
this function in terms of the Weierstrass coordinates of $X^+_{ns}(11)$.
The proof relies strongly on the fact that this elliptic curve has trivial rational $2$-torsion.
Through a different method not using this proof, in Section \ref{equation} we compute generators for the
function field of $X_{ns}(11)$ satisfying the equation given in \cite{DFGS}. They can be explicitly given
in terms of Siegel modular functions. For that purpose, some information on the modular curve $X_{ns}(11)$
and its cusps is recalled in Section \ref{secXnsp}.

\vspace{2truemm}

A basis of regular differentials for $X_{ns}(11)$ can be obtained from the equation in Section \ref{equation}
to produce a period matrix for the Jacobian of the curve. According to \cite{Chen}, this Jacobian is isogenous
over $\Q$ to the new part of the Jacobian of the genus $6$ modular curve $X_0(121)$. In Section \ref{Jacobian}
we show computationally that these two abelian varieties are in fact isomorphic over~$\Q$ and, nevertheless,
$X_0(121)$ is not a covering of $X_{ns}(11)$.

\vspace{2truemm}

\section{The modular curve $X_{ns}(11)$}\label{secXnsp}

Let $p$ be an odd prime. A non-split Cartan group of level $p$ is a subgroup of $\GL_2(\F_p)$
in the conjugacy class of
\begin{equation*}
\left\{\begin{pmatrix}m&\alpha\,n\\n&m\end{pmatrix}  \ \,\Big\vert\, \
(0,0)\neq(m,n)\in\F^2_{\!p}\right\}
\end{equation*}
for any non-square $\alpha$ in $\F^{^*}_{\!p}$. Such a matrix group and its normalizer in $\GL_2(\F_p)$
define, as shown in \cite{Lig} or \cite{Ma77}, respective canonical curves $X_{ns}(p)$ and $X_{ns}^+(p)$
defined over $\Q$, along with projections
$$\xymatrix{
 X(p)\, \ar@{->}[r]^{\!\!\!\!\!\pi} & \,X_{ns}(p)\, \ar@{->}[r]^{\pi^{+}} &  \,X^+_{ns}(p)\,
 \ar@{->}[r]^{j} & \,X(1),}
$$
where $X(p)$ is a model over $\Q$ for the classical modular curve attached to the kernel of the mod~$p$
reduction map $\SL_2(\Z)\!\longrightarrow\!\SL_2(\F_{\!p})$. The map $\pi^+$ has degree $2$ and is defined
over $\Q$, so that $X_{ns}^+(p)=X_{ns}(p)/\langle w\rangle$ for an involution $w$ of $X_{ns}(p)$ defined over $\Q$.
This involution is conjectured to be the the only non-trivial automorphism of $X_{ns}(p)$ whenever $p>11$.
As for the maps $\pi$ and $j$, they have degree $(p+1)/2$ and $p\,(p-1)/2$, respectively.
We refer to \cite{Baran}, where formulas for the genus and the number of cusps for $X_{ns}(p)$ and $X_{ns}^+(p)$ are given.

\vspace{2truemm}

An explicit description of the cusps is required for our purposes. There is a bijection between the set of cusps
of $X(p)$ and the quotient set of \,$\F^2_{\!p}\setminus\{(0,0)\}$ by the equivalence relation identifying
a pair~$(m,n)$ with the pair~$(-m,-n)$. This bijection comes from the reduction mod~$p$ of the coordinates,
taken to be coprime, of every point in~$\PP^1(\Q)$. The $(p^2-1)/2$ cusps thus obtained
map to the cusps of $X^+_{ns}(p)$, which are represented by the pairs of the form $(m,0)$ for $m=1,\dots,(p-1)/2$.
On the above canonical model for~$X^+_{ns}(p)$, the set of cusps is a Galois orbit defined over
the maximal real subfield of the $p$-th cyclotomic extension of $\Q$. Each of these cusps has two preimages
on $X_{ns}(p)$ which are switched by the involution $w$, and the set of cusps on $X_{ns}(p)$ is a Galois orbit
defined over the $p$-th cyclotomic extension of $\Q$ \cite{Se89}.

\vspace{2truemm}

From now on, we stick to the case $p=11$, namely to the only prime $p$ for which $X^+_{ns}(p)$
has genus $1$. As a matter of fact, $X^+_{ns}(11)$ is an elliptic curve.
A Weierstrass equation for $X^+_{ns}(11)$ was determined in \cite{Lig}:
\begin{equation}\label{xns+}
y^2\,+\,y\,=\,x^3\,-\,x^2\,-7x\,+\,10.
\end{equation}
It has conductor $121$ and label B1 in \cite{Cre}.
Its Mordell-Weil group is infinite and generated by the point $P\!:=\!(4,-6)$.
So $X_{ns}^+(11)$, as an algebraic curve, has infinitely many automorphisms defined over \,$\Q$.
They are generated by the elliptic involution
$$(x,y)\,\longmapsto\,(x,-1-y)$$\\[-20pt]
and the translation-by-$P$ map\\
$$(x,y)\,\longmapsto\,
\left(\dfrac{\,4x^2+x-2+11y\,}{\left(x-4\right)^2}\,, \,
\dfrac{\,\left(2x^2+17x-34+11y\right)\left(1-3x\right)}{\left(x-4\right)^3}\right)\!.$$
Among the corresponding infinitely many pairs of functions in $\Q\big(X_{ns}^+(11)\big)$ satisfying (\ref{xns+}),
in Section \ref{equation} we compute functions for which the projection
$$j\ \colon \xymatrix{X^+_{ns}(11) \ \ar@{->}[r] & \ X(1)}$$\\[-10pt]
is given by the rational function
\begin{footnotesize}
\begin{eqnarray}\label{j}
\begin{split}
j(x,y) \,\,=\,\,  &
\left(7x^2+16x-44+\left(x+18\right)y\right)^2\left(4x^3+x^2-24x-11-\left(x^2+3x+5\right)y\right)^{11}\\
&
\left(11\left(y-5\right)\left(x^2+3x-6\right)\!\left(3x^2-3x-14-\left(2x+3\right)y\right)\right)^3\\
&
\left(\left(12x^3+28x^2-41x-62+\left(3x^2+20x+37\right)y\right)\!\left(x^3+4x^2+x+22-\left(3x-1\right)y\right)\right)^3\\
& \left(x+2\right)^{-12}\left(x-4\right)^{-14}\left(x^5-9x^4-16x^3+53x^2+37x-23\right)^{-11}
\end{split}
\end{eqnarray}
\end{footnotesize}
\!\!obtained through the above translation from the expression for $j$\, in Section 2.2 of \cite{Hal}.
This fixes the $j$-invariants of the points on $X^+_{ns}(11)$ corresponding to the origin $O$ of (\ref{xns+})
and to the point~$P$\!, \,namely \,$j(O)=2^3\,3^3\,11^3$\, and  \,$j(P)=-2^{15}\,5^3\,3$.

\vspace{2truemm}

The bielliptic curve $X_{ns}(11)$ has genus $4$. We note that Proposition~1 in~\cite{Ogg} implies that it is not hyperelliptic.
The equation derived in \cite{DFGS} for this curve, as a degree $2$ covering of the elliptic curve $X^+_{ns}(11)$,
is
\begin{equation}\label{xns}
t^2 \,=\, -(4x^3\,+\,7x^2\,-6x\,+\,19).
\end{equation}
In particular, the curve admits an \emph{exotic} involution $\varrho$ inducing the elliptic involution on $X^+_{ns}(11)$.
The goal of \cite{DFGS} is proving that $w$ and $\varrho$ generate the full automorphism
group of $X_{ns}(11)$.

\vspace{2truemm}

\section{The function field of $X_{ns}(11)$}\label{equation}

In this section we compute generators $X, Y\!, T$ for the function field of $X_{ns}(11)$ satisfying
the equations (\ref{xns+}) and (\ref{xns}).
They can be made explicit from certain modular functions of level~$11$ with cuspidal divisor.
We must first gather the required ingredients at the starting point of our computations; the main reference is \cite{Ku-Lang}.

\vspace{2truemm}

Let us fix the $11$-th root of unity \,$\zeta\!:=\!e^{\,2\pi i/11}$ together with the generator
\mbox{\,$\epsilon:=\zeta+\zeta^{-1}$} for the maximal real subfield of \,$\Q(\zeta)$.
For $\tau$ in the complex upper half-plane, let us take\, $q_{\scriptscriptstyle\ast}\!:=\!e^{\,2\pi i\tau /11}$ as
local parameter of level~$11$ \emph{at infinity}.
The set of modular functions of level~$11$ whose Fourier \mbox{$q_{\scriptscriptstyle\ast}$--expansion} has coefficients in \,$\Q(\zeta)$
can be naturally identified with the function field of the modular curve~$X(11)$ over that cyclotomic extension.
Any such a function can be uniquely \emph{normalized}\, by a suitable constant
so that the first non-zero coefficient of its Fourier $q_{\scriptscriptstyle\ast}$--expansion \mbox{is $1$.}

\vspace{2truemm}

To every cusp of $X(11)$, represented by a pair $(m,n)$ as in Section \ref{secXnsp},
one can attach the Siegel function
$$S_{\scriptscriptstyle (m,n)} :=
q_{\scriptscriptstyle\ast}^{\,11/2\,B_2(m/11)}\,
\big(1 - \zeta^n q_{\scriptscriptstyle\ast}^{\,m}\big)\prod_{k\geq 1}
\big(1-\zeta^n q_{\scriptscriptstyle\ast}^{\,11k+m}\big)
\big(1-\zeta^{-n} q_{\scriptscriptstyle\ast}^{\,11k-m}\big),$$
where \,$B_2(x)=x^2-x+1/6$\, is the second Bernoulli polynomial.
The power $S_{\scriptscriptstyle (m,n)}^{^{\scriptscriptstyle 132}}$ becomes
a function on $X(11)$ whose divisor has support in the set of cusps and can be explicitly given: the order
at a cusp represented by a pair $(m',n')$ is \,$726\,B_2\big((m\,m'+n\,n')/11\big)$, where one regards
the sum \,$m\,m'+n\,n'$ reduced mod $11$. The product of all Siegel funcionts $S_{\scriptscriptstyle (m,n)}$
is a constant in $\Q(\zeta)$. As a matter of fact, every product of Siegel functions whose
divisor has integer coefficients lies in the function field of $X(11)$ over $\Q(\zeta)$.
Conversely, every \emph{modular unit} on $X(11)$, that is, every non-constant modular function of level $11$
with cuspidal divisor can be explicitly written, up to a non-zero multiplicative constant,
as a product of Siegel functions.

\vspace{2truemm}

Finally, we need some more notation to proceed further. As in Section \ref{secXnsp}, $\pi$ \,and\, $\pi^+$ stand for the canonical projections
$$\pi\,\colon X(11) \longrightarrow X_{ns}(11),\qquad \qquad
\pi^+\,\colon X_{ns}(11) \longrightarrow X^+_{ns}(11).$$
For $k=1,\dots,5$, let $P_{\!\scriptscriptstyle k}$ denote the image by $\pi$ of the
cusp represented by the pair $(k,0)$. The cuspidal points on $X_{ns}(11)$ are then
$P_{\!\scriptscriptstyle 1},\dots,P_{\!\scriptscriptstyle 5},
w P_{\!\scriptscriptstyle 1},\dots,w P_{\!\scriptscriptstyle 5}$.
They all have ramification index~$11$ over $X(1)$, are defined over $\Q(\zeta)$ and are transitively
permuted by the Galois group of this extension over $\Q$.
The cuspidal points on $X^+_{ns}(11)$ are
\,$\pi^+(P_{\!\scriptscriptstyle 1}),\dots,\pi^+(P_{\!\scriptscriptstyle 5})$ and their
minimal field of definition is $\Q(\epsilon)$.

\vspace{2truemm}

The first step in our computations is an application of the results in \cite{Ku-Lang}.

\begin{lemma}\label{K-L}
There are exactly two normalized modular units $\widetilde X, \widetilde Y$\! on~$X_{ns}^+(11)$ having degree at most~$3$
and only a pole at \,$\pi^+(P_{\!\scriptscriptstyle 1})$.
There are exactly two normalized modular units $\widetilde U, \widetilde V$\! on~$X_{ns}(11)$ which
are not functions on $X_{ns}^+(11)$ and have only poles at $P_{\!\scriptscriptstyle 1}$
and~\,$w P_{\!\scriptscriptstyle 1}$ of order $5$.
Specifically, the functions \,$\widetilde X$, $\widetilde Y$, $\widetilde U$, $\widetilde V$ have divisors
$$\begin{array}{ccl}
\dv(\widetilde X) & = & \left(\pi^+(P_{\!\scriptscriptstyle 3})\right)
\,+\,\left(\pi^+(P_{\!\scriptscriptstyle 5})\right)
\,-\,2\left(\pi^+(P_{\!\scriptscriptstyle 1})\right),\\[10pt]
\dv(\widetilde Y) & = & \left(\pi^+(P_{\!\scriptscriptstyle 3})\right)
\,+\,2\left(\pi^+(P_{\!\scriptscriptstyle 2})\right)
\,-\,3\left(\pi^+(P_{\!\scriptscriptstyle 1})\right),\\[10pt]
\dv(\widetilde U) & = & 4\left(P_{\!\scriptscriptstyle 2}\right)\,+\,\left(P_{\!\scriptscriptstyle 4}\right)
\,+\,5\left(P_{\!\scriptscriptstyle 3}\right)
\,-\,5\left(P_{\!\scriptscriptstyle 1}\right)\,-\,5\left(w P_{\!\scriptscriptstyle 1}\right),\\[10pt]
\dv(\widetilde V) & = & 4\left(w P_{\!\scriptscriptstyle 2}\right)\,+\,\left(w P_{\!\scriptscriptstyle 4}\right)
\,+\,5\left(w P_{\!\scriptscriptstyle 3}\right)
\,-\,5\left(P_{\!\scriptscriptstyle 1}\right)\,-\,5\left(w P_{\!\scriptscriptstyle 1}\right),
\end{array}$$
and they are respectively given, up to the product by elements in $\Q(\epsilon)$, by
$$
\!\!\!g_{\scriptscriptstyle 5}\,h_{\scriptscriptstyle 5}
\,,\qquad\quad
\dfrac{\,g_{\scriptscriptstyle 5}\,h_{\scriptscriptstyle 5}\,}{\,g_{\scriptscriptstyle 3}\,h_{\scriptscriptstyle 3}\,}
\,,\qquad\quad
\dfrac{\,g_{\scriptscriptstyle 5}^2\,h_{\scriptscriptstyle 5}\,}{\,h_{\scriptscriptstyle 2}\,g_{\scriptscriptstyle 3}\,h_{\scriptscriptstyle 3}^2\,g_{\scriptscriptstyle 4}\,}
\,,\qquad\quad
\dfrac{\,g_{\scriptscriptstyle 5}\,h_{\scriptscriptstyle 5}^2\,}{\,g_{\scriptscriptstyle 2}\,g_{\scriptscriptstyle 3}^2\,h_{\scriptscriptstyle 3}\,h_{\scriptscriptstyle 4}\,}\,,
$$
where $g_{\scriptscriptstyle k}$ and\, $h_{\scriptscriptstyle k}$ stand for the products of Siegel functions \,$S_{\scriptscriptstyle (m,n)}$\, for~\,$(m,n)$ running over
a system of representatives for the set of cusps \,$\pi^{-1}(P_{\!\scriptscriptstyle k})$ and
\,$\pi^{-1}(w P_{\!\scriptscriptstyle k})$, respectively.
\end{lemma}

\vspace{2truemm}

As it can be checked from their Fourier $q_{\scriptscriptstyle\ast}$-expansions,
the functions in Lemma \ref{K-L} satisfy the relations
\begin{footnotesize}
\begin{equation*}\label{eq1}
\!\!\!\widetilde Y^2-2\,\epsilon\,\widetilde X\,\widetilde Y
+\left(2 + 6 \epsilon + 2 \epsilon^2 - 2 \epsilon^3 - \epsilon^4\right)\!\widetilde Y
= \widetilde X^3 - 2\left(3 \epsilon + 3 \epsilon^2 - \epsilon^3 - \epsilon^4\right)\!\widetilde X^2
+ \left(3 + 11 \epsilon + 5 \epsilon^2 - 4 \epsilon^3 - 2 \epsilon^4\right)\!\widetilde X
\end{equation*}
\end{footnotesize}
\!\!and
\begin{footnotesize}
\begin{equation}\label{eq2}
\widetilde U\,\widetilde V \,=\,
\widetilde Y^2\left(\widetilde X^2 \,-\, \left(1-3\epsilon+\epsilon^3\right)\widetilde Y
\,+\, \left(1-2\epsilon-3\epsilon^2+\epsilon^3+\epsilon^4\right)\widetilde X\right)\!.
\end{equation}
\end{footnotesize}

\vspace{1truemm}

\begin{prop}\label{prop1}
With the notation in \,{\rm Lemma \ref{K-L}}, consider the functions
\begin{small}
$$\begin{array}{ccl}
\widehat X & \!\!:=\!\! & 5 - 5 \epsilon - 7 \epsilon^2 + 2 \epsilon^3 +
 2 \epsilon^4 + \left(12 - 7 \epsilon - 12 \epsilon^2 + 3 \epsilon^3 +
    3 \epsilon^4\right)\!\widetilde X, \\[10pt]
\widehat Y & \!\!:=\!\! & \!-12 + 3 \epsilon + 10 \epsilon^2 - \epsilon^3 - 2 \epsilon^4
-\left(13 - 7 \epsilon - 20 \epsilon^2 + 5 \epsilon^3 + 6 \epsilon^4\right)\!\widetilde X
-\left(46 - 19 \epsilon - 57 \epsilon^2 + 7 \epsilon^3 + 13 \epsilon^4\right)\!\widetilde Y\!
\end{array} $$
\end{small}
\!and the constants
\begin{small}
$$
\begin{array}{ll}
\alpha_{\scriptscriptstyle 2}:=11 \epsilon + 6 \epsilon^2 - 4 \epsilon^3 - 2 \epsilon^4, \qquad &
\gamma_{\scriptscriptstyle 3}:=-4 - 26 \epsilon - 19 \epsilon^2 + 10 \epsilon^3 + 7 \epsilon^4, \\[5pt]

\alpha_{\scriptscriptstyle 1}:= 21 + 66 \epsilon + 25 \epsilon^2 - 24 \epsilon^3 - 12 \epsilon^4, \qquad &
\gamma_{\scriptscriptstyle 2}:= -200 - 722 \epsilon - 313 \epsilon^2 + 259 \epsilon^3 + 133 \epsilon^4, \\[5pt]

\alpha_{\scriptscriptstyle 0}:=17 - 103 \epsilon - 61 \epsilon^2 + 38 \epsilon^3 + 21 \epsilon^4, \qquad &
\gamma_{\scriptscriptstyle 1}:=\phantom{+}208 + 1062 \epsilon + 633 \epsilon^2 - 396 \epsilon^3 - 245 \epsilon^4,\\[5pt]

\,\delta_{\scriptscriptstyle 1}:= 77 + 242 \epsilon + 99 \epsilon^2 - 88 \epsilon^3 - 44 \epsilon^4, \qquad &
\gamma_{\scriptscriptstyle 0}:=\phantom{+}534 + 1499 \epsilon + 425 \epsilon^2 - 516 \epsilon^3 - 210 \epsilon^4,\\[5pt]

\,\delta_{\scriptscriptstyle 0}:=  187 + 517 \epsilon + 264 \epsilon^2 - 187 \epsilon^3 - 110 \epsilon^4, \qquad &
\, \beta\,:=-7 - 52 \epsilon - 38 \epsilon^2 + 20 \epsilon^3 + 14 \epsilon^4 .
\end{array}
$$
\end{small}
\!\!Then,
$$
X=\dfrac{\,\alpha_{\scriptscriptstyle 2}\widehat X^2+\alpha_{\scriptscriptstyle 1}\widehat X+\alpha_{\scriptscriptstyle 0}
+\beta\,\widehat Y\,}
{\,\big(\widehat X - 11 \epsilon - 6 \epsilon^2 + 4 \epsilon^3 +2 \epsilon^4\big)^{\!2}\,} \qquad
\mathrm{and} \qquad
Y=\dfrac{\,\gamma_{\scriptscriptstyle 3}\widehat X^3+\gamma_{\scriptscriptstyle 2}\widehat X^2+
\gamma_{\scriptscriptstyle 1}\widehat X+\gamma_{\scriptscriptstyle 0}+
\big(\delta_{\scriptscriptstyle 1}\widehat X+\delta_{\scriptscriptstyle 0}\big)\widehat Y\,}
{\big(\widehat X - 11 \epsilon - 6 \epsilon^2 + 4 \epsilon^3 +2 \epsilon^4\big)^{\!3}}
$$
are the generators of the function field \,$\Q\big(X^+_{ns}(11)\big)$ satisfying the relations {\rm (\ref{xns+})} and
{\rm (\ref{j})}.
Their values at the cusps are
$$
\begin{array}{ll}
X(P_{\!\scriptscriptstyle 1})= \phantom{+}11 \epsilon + 6 \epsilon^2 - 4 \epsilon^3 - 2 \epsilon^4, & \quad\qquad
Y(P_{\!\scriptscriptstyle 1})= -4 - 26 \epsilon - 19 \epsilon^2 + 10 \epsilon^3 + 7 \epsilon^4,\\[5pt]
X(P_{\!\scriptscriptstyle 2})= \phantom{+}10 - 3 \epsilon - 10 \epsilon^2 + \epsilon^3 + 2 \epsilon^4, & \quad\qquad
Y(P_{\!\scriptscriptstyle 2})= -29 + 13 \epsilon + 37 \epsilon^2 - 5 \epsilon^3 - 9 \epsilon^4,\\[5pt]
X(P_{\!\scriptscriptstyle 3})= -3 - 4 \epsilon + 11 \epsilon^2 + \epsilon^3 - 3 \epsilon^4,
& \quad\qquad
Y(P_{\!\scriptscriptstyle 3})= \phantom{+}2 + 17 \epsilon - 21 \epsilon^2 - 4 \epsilon^3 + 6 \epsilon^4,\\[5pt]
X(P_{\!\scriptscriptstyle 4})= \phantom{+}2 - 2 \epsilon - 5 \epsilon^2 + 2 \epsilon^3 + 2 \epsilon^4, & \quad\qquad
Y(P_{\!\scriptscriptstyle 4})= -\epsilon - 2 \epsilon^2 - 3 \epsilon^3 - \epsilon^4,\\[5pt]
X(P_{\!\scriptscriptstyle 5})= -2 \epsilon - 2 \epsilon^2 + \epsilon^4, & \quad\qquad
Y(P_{\!\scriptscriptstyle 5})= \phantom{+}1 - 3 \epsilon + 5 \epsilon^2 + 2 \epsilon^3 - 3 \epsilon^4,
\end{array}
$$
where \,$X(Q)$, $Y(Q)$ is an abuse of notation for \,$X\!\left(\pi^+(Q)\right)$, $Y\!\left(\pi^+(Q)\right)$,
respectively. The first coefficients of their Fourier $q_{\scriptscriptstyle\ast}$-expansions are as follows{\rm :}
$$\begin{array}{c||r|r}
q_{\scriptscriptstyle\ast}^n & X\qquad\qquad & Y\qquad\qquad\qquad  \\ \hline\hline
& & \\[-10pt]
\,1\,\, &
11 \epsilon + 6 \epsilon^2 - 4 \epsilon^3 -   2 \epsilon^4 &-4 - 26 \epsilon - 19 \epsilon^2 + 10 \epsilon^3+7 \epsilon^4
\\[4pt]
\,\,q_{\scriptscriptstyle\ast}\,\, & 12 + 54 \epsilon + 22 \epsilon^2 - 19 \epsilon^3 - 9 \epsilon^4   &-62 - 239 \epsilon - 115 \epsilon^2 + 86 \epsilon^3 + 47 \epsilon^4  \\[4pt]
q_{\scriptscriptstyle\ast}^2 & 58 + 237 \epsilon +    113 \epsilon^2 - 86 \epsilon^3 - 46 \epsilon^4  &-352 -    1463 \epsilon - 704 \epsilon^2 + 528 \epsilon^3 + 286 \epsilon^4   \\[4pt]
q_{\scriptscriptstyle\ast}^3 & 232 + 968 \epsilon + 477 \epsilon^2 -   351 \epsilon^3 - 192 \epsilon^4 &-1763 - 7441 \epsilon -    3640 \epsilon^2 + 2691 \epsilon^3 + 1469 \epsilon^4\\[4pt]
q_{\scriptscriptstyle\ast}^4 & 889 + 3760 \epsilon + 1854 \epsilon^2 - 1362 \epsilon^3 -   746 \epsilon^4  &-8158 - 34317 \epsilon - 16835 \epsilon^2 + 12417 \epsilon^3 + 6790 \epsilon^4
\end{array}
$$
\end{prop}

\vspace{1truemm}

\begin{proof}
From the Weierstrass equation given in (\ref{eq2}), $\widehat X$ and $\widehat Y$ can be checked to satisfy
$$\widehat Y^2\,+\,\widehat Y\,=\,\widehat X^3\,-\,\widehat X^2\,-7\,\widehat X\,+\,10.$$
They are functions in \,$\Q(\zeta)\!\big(X^+_{ns}(11)\big)$
with a pole at the cusp \,$\pi^+(P_{\scriptscriptstyle 1})$ of order $2$ and $3$, respectively.
This means that their relation with $X$ and $Y$ must be given, up to the elliptic involution,
by translation by $\pi^+(P_{\scriptscriptstyle 1})$, namely,
$$
(X,\,Y)\,=\,(\widehat X,\,\widehat Y) \,+\,
\left(X(P_{\!\scriptscriptstyle 1}),\,Y(P_{\!\scriptscriptstyle 1})\right)
\qquad
\mathrm{or}
\qquad
(X,\,Y)\,=\,(\widehat X,\,-1-\widehat Y) \,+\,
\left(X(P_{\!\scriptscriptstyle 1}),\,Y(P_{\!\scriptscriptstyle 1})\right),
$$
where the addition sign stands for the group law in the Weierstrass equation (\ref{xns+}).
The exact relation turns out to be of the second type and holds for
$$X(P_{\!\scriptscriptstyle 1})=11 \epsilon + 6 \epsilon^2 - 4 \epsilon^3 - 2 \epsilon^4, \qquad\quad\qquad
Y(P_{\!\scriptscriptstyle 1})=-4 - 26 \epsilon - 19 \epsilon^2 + 10 \epsilon^3 + 7 \epsilon^4.$$
Indeed, it is uniquely determined from the values that the functions $X$, $Y$, $\widehat X$, $\widehat Y$ take at
the other cuspidal points. On the one hand, each of these points is a zero of
either $\widetilde  X$, $\widetilde Y$ or \,$\widetilde U\,\widetilde V$\!,
\,so that the relations in (\ref{eq2}) lead to the following values:
$$
\begin{array}{ll}
\widehat X(P_{\!\scriptscriptstyle 2})= -2 + \epsilon + 5 \epsilon^2 - \epsilon^4, & \quad\qquad
\widehat Y(P_{\!\scriptscriptstyle 2})= -4 + 2 \epsilon + 5 \epsilon^2 - 2 \epsilon^3 - 2 \epsilon^4,\\[5pt]
\widehat X(P_{\!\scriptscriptstyle 3})= \phantom{+}5 - 5 \epsilon - 7 \epsilon^2 + 2 \epsilon^3 + 2 \epsilon^4,
& \quad\qquad
\widehat Y(P_{\!\scriptscriptstyle 3})= -12 + 3 \epsilon + 10 \epsilon^2 - \epsilon^3 - 2 \epsilon^4,\\[5pt]
\widehat X(P_{\!\scriptscriptstyle 4})= \phantom{+}1 - 2 \epsilon + 3 \epsilon^2 - \epsilon^4, & \quad\qquad
\widehat Y(P_{\!\scriptscriptstyle 4})= -2 - 4 \epsilon + 11 \epsilon^2 + \epsilon^3 - 3 \epsilon^4,\\[5pt]
\widehat X(P_{\!\scriptscriptstyle 5})= \phantom{+}5 - 5 \epsilon - 7 \epsilon^2 + 2 \epsilon^3 + 2 \epsilon^4, & \quad\qquad
\widehat Y(P_{\!\scriptscriptstyle 5})= \phantom{+}11 - 3 \epsilon - 10 \epsilon^2 + \epsilon^3 + 2 \epsilon^4.
\end{array}
$$
On the other hand, the cusps on $X_{ns}^+(11)$ have ramification index $11$ over $X(1)$ and are transitively
permuted by the Galois group of $\Q(\epsilon)$ over $\Q$. Thus, the expression (\ref{j}) yields
$$
X(P_{\!\scriptscriptstyle k})=-2 z - 2 z^2 + z^4, \quad\qquad\qquad\quad
Y(P_{\!\scriptscriptstyle k})=1 - 3 z + 5 z^2 + 2 z^3 - 3 z^4,
$$
for $z$ lying in the set
$$\left\{\epsilon,\,\,-2+\epsilon^2,\,\,2-4\epsilon^2+\epsilon^4,\,\,-3\epsilon+\epsilon^3,
\,\,-1+2\epsilon+3\epsilon^2-\epsilon^3-\epsilon^4\right\}$$
of Galois conjugates of $\epsilon$. The way in which $z$ runs over this set for \,$k=1,\dots,5$\,
is the only one making the values
$\widehat X(P_{\!\scriptscriptstyle k}), \widehat Y(P_{\!\scriptscriptstyle k})$
match with  the values $X(P_{\!\scriptscriptstyle k}), Y(P_{\!\scriptscriptstyle k})$ through one of the two possible relations above.
\end{proof}

\vspace{1truemm}

\begin{rem}\label{reciproc}
{\rm The relations
$$
(X,\,Y)\,=\,(\widehat X,\,-1-\widehat Y) \,+\,
\left(X(P_{\!\scriptscriptstyle 1}),\,Y(P_{\!\scriptscriptstyle 1})\right)
\qquad
\mathrm{and}
\qquad
(\widehat X,\,\widehat Y)\,=\,(X,\,-1-Y) \,+\,
\left(X(P_{\!\scriptscriptstyle 1}),\,Y(P_{\!\scriptscriptstyle 1})\right)
$$
are equivalent. This means that the expressions for $X$ and $Y$ as rational functions on $\widehat X$, $\widehat Y$
are the same as the expressions for $\widehat X$ and $\widehat Y$\!, respectively, as rational functions on $X$, $Y$\!.\,
These expressions can thus be seen as an involution on $X^+_{ns}(11)$ defined over $\Q(\epsilon)$.
}\end{rem}

\vspace{1truemm}

\begin{prop}\label{prop2} With the notation in \,{\rm Lemma \ref{K-L}} and \,{\rm Proposition \ref{prop1}}, consider
the function
$$\widehat T  \,:=\,  \widetilde U \,+\, \dfrac{1}{\,11\,}\left(8 + 17 \epsilon - 21 \epsilon^2 - 4 \epsilon^3
+ 6 \epsilon^4\right)\widetilde V$$
and let \,$\overline T$ be the trace of
\,\,$\big(X - 11 \epsilon - 6 \epsilon^2 + 4 \epsilon^3 + 2 \epsilon^4\big)^{\!5}\,\,\widehat T\,$
from the function field \,$\Q(\zeta)\!\big(X_{ns}(11)\big)$ to the subextension
\,$\Q\!\left(\sqrt{-11}\,\right)\!\!\big(X_{ns}(11)\big)$.
Then, the function
$$T \,:=\, \dfrac{\,\overline T\,}{\,\sqrt{-11}\,
\big(1458X^3 + 12640X^2 - 56666X + 59141 - (107X^2 + 7296X - 5508)Y\big)\,}$$
generates \,$\Q\big(X_{ns}(11)\big)$ over \,$\Q\big(X^+_{ns}(11)\big)$ and satisfies the equation\, {\rm (\ref{xns})}
along with $X$\!.
The first coefficients of its Fourier $q_{\scriptscriptstyle\ast}$-expansion are as follows{\rm :}
$$\begin{array}{c||r}
 q_{\scriptscriptstyle\ast}^n &  T/\sqrt{-11} \qquad\qquad\qquad\\ \hline\hline
 & \\[-10pt]
 \,1\,\, & \ -5 -17\epsilon -10\epsilon^2 + 6\epsilon^3 +4\epsilon^4 \ \\[4pt]
  \,\,q_{\scriptscriptstyle\ast}\,\, & \ -34 -152\epsilon -78\epsilon^2 +55\epsilon^3 +31\epsilon^4 \ \\[4pt]
   q_{\scriptscriptstyle\ast}^2 & \ -214 -903\epsilon -441\epsilon^2 +326\epsilon^3 + 178\epsilon^4 \ \\[4pt]
   q_{\scriptscriptstyle\ast}^3 & \ -1088 -4550\epsilon -2215\epsilon^2 + 1645\epsilon^3 + 896\epsilon^4 \
\end{array}
$$
\end{prop}

\vspace{1truemm}

\begin{proof}
The functions $\widetilde V$ and \,$\widetilde Uw$\, have the same divisor, so
\,$\widetilde Uw=a\,\widetilde V$\, for some constant \,$a$. This constant can be computed
using that \,$\widetilde U + a\,\widetilde V$\!, which is a function in \,$\Q(\zeta)\!\big(X^+_{ns}(11)\big)$
with only a pole of order $5$ at the cusp \,$\pi^+(P_{\scriptscriptstyle 1})$,
must be a polynomial in $\widetilde X$ and $\widetilde Y$\!.
Specifically, from the Fourier $q_{\scriptscriptstyle\ast}$-expansions of these functions one obtains
\begin{align}\label{uvxy}
\widetilde U + a\,\widetilde V \,=\, b_{\scriptscriptstyle 2}\widetilde X^2\,+\,b_{\scriptscriptstyle 1}\widetilde X
\,+\,b_{\scriptscriptstyle 0} \,+\, (c_{\scriptscriptstyle 1}\widetilde X\,+\,c_{\scriptscriptstyle 0})\widetilde Y
\end{align}
for
\begin{small}
$$\begin{array}{ll}
\ 11\,a\,=-8-17\epsilon+21\epsilon^2+4\epsilon^3-6\epsilon^4, \qquad &
\quad 11\,b_{\scriptscriptstyle 2}=-21+9\epsilon-15\epsilon^2+5\epsilon^3+9\epsilon^4,\\[5pt]
\ 11\,c_{\scriptscriptstyle 1}=\phantom{+}3-17\epsilon+21\epsilon^2+4\epsilon^3-6\epsilon^4, \qquad &
\quad 11\,b_{\scriptscriptstyle 1}=-13+4\epsilon+52\epsilon^2-10\epsilon^3-18\epsilon^4,\\[5pt]
\ 11\,c_{\scriptscriptstyle 0}=-25-71\epsilon+45\epsilon^2+29\epsilon^3-5\epsilon^4, \qquad &
\quad 11\,b_{\scriptscriptstyle 0}=\phantom{+}17+76\epsilon+31\epsilon^2-25\epsilon^3-12\epsilon^4.\\[3pt]
\end{array}$$
\end{small}
\!\!\!The function \,$\widehat T\!=\!\widetilde U-a\,\widetilde V$ keeps the same polar part as $\widetilde U$,
\!$\widetilde V$ and satisfies \,$\widehat Tw=-\widehat T$\!,
so its square must also be a polynomial in $\widetilde X, \widetilde Y$\, hence in $\widehat X, \widehat Y$\!.
The relations given in Proposition~\ref{prop1} and Remark \ref{reciproc} between $\widehat X,\widehat Y$ and $X,Y$
lead to an identity of the form
\begin{small}
\begin{equation}\label{almost_equation}
\left(\big(X - 11 \epsilon - 6 \epsilon^2 + 4 \epsilon^3 + 2 \epsilon^4\big)^{\!5}\,\,\widehat T\,\right)^{\!2}=\,
\dfrac{1}{\,11\,}\left(4X^3+7X^2-6X+19\right)\!\big(f_1(X) \,+\, f_2 (X)\,Y\big)^2
\end{equation}
\end{small}
\!\!for certain polynomials \,$f_1(x), f_2(x)$ in $\Z[\epsilon][x]$
of respective degrees $3$ and $2$ (see Remark \ref{norm}).

\vspace{2truemm}

Let us then take the trace \,$\overline T$\, of
\,$\big(X - 11 \epsilon - 6 \epsilon^2 + 4 \epsilon^3 + 2 \epsilon^4\big)^{\!5}\,\,\widehat T\,$ to
\,$\Q\!\left(\sqrt{-11}\,\right)\!\!\big(X_{ns}(11)\big)$. Since~$X$ is defined over $\Q$, it suffices
to compute the conjugate functions \,${^{\sigma^i}}\widetilde U$ and \,${^{\sigma^i}}\widetilde V$\!
\,for $i=1,\dots,4$, where~$\sigma$ is the automorphism of the Galois extension
\,$\Q(\zeta)/\Q\!\left(\sqrt{-11}\,\right)$ sending $\zeta$ to~$\zeta^9$ \,hence \,$\epsilon$\, to \,$\epsilon^2-2$.
\,To that end, one can first derive \,${^{\sigma^i}\!}\widetilde X$\, and \,${^{\sigma^i}}\widetilde Y$\,
as rational functions on $X$ and $Y$ with coefficients in $\Q(\epsilon)$. This is a straightforward application
of Proposition \ref{prop1} and Remark \ref{reciproc}.

\vspace{2truemm}

Note that the values at the cusps given in Proposition \ref{prop1} imply\\[-6pt]
\begin{small}
$$\pi^+({^\sigma\!}P_{\scriptscriptstyle 1})=\pi^+(P_{\scriptscriptstyle 4}), \ \
\pi^+({^\sigma\!}P_{\scriptscriptstyle 4})=\pi^+(P_{\scriptscriptstyle 5}), \ \
\pi^+({^\sigma\!}P_{\scriptscriptstyle 5})=\pi^+(P_{\scriptscriptstyle 2}), \ \
\pi^+({^\sigma\!}P_{\scriptscriptstyle 2})=\pi^+(P_{\scriptscriptstyle 3}), \ \
\pi^+({^\sigma\!}P_{\scriptscriptstyle 3})=\pi^+(P_{\scriptscriptstyle 1}).
$$
\end{small}

\vspace{-5truemm}

\noindent Of course, this agrees with the formula (3.6.1) in \cite{Lig}, and the same permutation
can be checked to hold for the cusps $P_{\scriptscriptstyle 1},\dots,P_{\scriptscriptstyle 5}$.
Then, the information on Siegel functions leading to Lemma \ref{K-L} yields these expressions:
\begin{small}
$${^{\sigma}}\widetilde U = \nu_{\scriptscriptstyle 1}\dfrac{\,g_{\scriptscriptstyle 4}^2\,h_{\scriptscriptstyle 4}\,}{\,h_{\scriptscriptstyle 5}\,g_{\scriptscriptstyle 2}\,h_{\scriptscriptstyle 2}^2\,g_{\scriptscriptstyle 1}\,}\,,\,
\qquad
{^{\sigma^2}}\widetilde U = \nu_{\scriptscriptstyle 2}\dfrac{\,g_{\scriptscriptstyle 1}^2\,h_{\scriptscriptstyle 1}\,}{\,h_{\scriptscriptstyle 4}\,g_{\scriptscriptstyle 5}\,h_{\scriptscriptstyle 5}^2\,g_{\scriptscriptstyle 3}\,}\,,\,
\qquad
{^{\sigma^3}}\widetilde U = \nu_{\scriptscriptstyle 3}\dfrac{\,g_{\scriptscriptstyle 3}^2\,h_{\scriptscriptstyle 3}\,}{\,h_{\scriptscriptstyle 1}\,g_{\scriptscriptstyle 4}\,h_{\scriptscriptstyle 4}^2\,g_{\scriptscriptstyle 2}\,}\,,\,
\qquad
{^{\sigma^4}}\widetilde U = \nu_{\scriptscriptstyle 4}\dfrac{\,g_{\scriptscriptstyle 2}^2\,h_{\scriptscriptstyle 2}\,}{\,h_{\scriptscriptstyle 3}\,g_{\scriptscriptstyle 1}\,h_{\scriptscriptstyle 1}^2\,g_{\scriptscriptstyle 5}\,}\,,
$$
$$
{^{\sigma}}\widetilde V = \vartheta_{\scriptscriptstyle 1}\dfrac{\,g_{\scriptscriptstyle 4}\,h_{\scriptscriptstyle 4}^2\,}{\,g_{\scriptscriptstyle 5}\,g_{\scriptscriptstyle 2}^2\,h_{\scriptscriptstyle 2}\,h_{\scriptscriptstyle 1}\,}\,,
\qquad
{^{\sigma^2}}\widetilde V = \vartheta_{\scriptscriptstyle 2}\dfrac{\,g_{\scriptscriptstyle 1}\,h_{\scriptscriptstyle 1}^2\,}{\,g_{\scriptscriptstyle 4}\,g_{\scriptscriptstyle 5}^2\,h_{\scriptscriptstyle 5}\,h_{\scriptscriptstyle 3}\,}\,,
\qquad
{^{\sigma^3}}\widetilde V = \vartheta_{\scriptscriptstyle 3}\dfrac{\,g_{\scriptscriptstyle 3}\,h_{\scriptscriptstyle 3}^2\,}{\,g_{\scriptscriptstyle 1}\,g_{\scriptscriptstyle 4}^2\,h_{\scriptscriptstyle 4}\,h_{\scriptscriptstyle 2}\,}\,,
\qquad
{^{\sigma^4}}\widetilde V = \vartheta_{\scriptscriptstyle 4}\dfrac{\,g_{\scriptscriptstyle 2}\,h_{\scriptscriptstyle 2}^2\,}{\,g_{\scriptscriptstyle 3}\,g_{\scriptscriptstyle 1}^2\,h_{\scriptscriptstyle 1}\,h_{\scriptscriptstyle 5}\,}\,,\\[5pt]
$$
\end{small}
\!\!where \,$\nu_{\scriptscriptstyle i}$ and \,$\vartheta_{\scriptscriptstyle i}$
are constants which can be determined from the conjugate functions
\,${^{\sigma^i}\!}\widetilde X, {^{\sigma^i}}\widetilde Y$ and the relation (\ref{uvxy}).
They are displayed in the following table:\\[-2pt]
$$\begin{array}{c||c|c}
\ i \ & \ \nu_{\scriptscriptstyle i} \ & \ \vartheta_{\scriptscriptstyle i} \ \\
\hline
\hline
& & \\[-10pt]
\ 1 \ & \ 2+7\epsilon-\epsilon^2-8\epsilon^3-3\epsilon^4 \ & \ \epsilon^2-\epsilon^3-\epsilon^4 \ \\[2pt]
\hline
& & \\[-10pt]
\ 2 \ & \ -1+2\epsilon^2-\epsilon^3 \  & \ 4\epsilon+\epsilon^2-2\epsilon^3 \ \\[2pt]
\hline
& & \\[-10pt]
\ 3 \ & \ 10-3\epsilon-12\epsilon^2+\epsilon^3+3\epsilon^4 \ & \ -27+15\epsilon+33\epsilon^2-8\epsilon^3-6\epsilon^4 \ \\[2pt]
\hline
& & \\[-10pt]
\ 4 \ & \ -2\epsilon+10\epsilon^2-3\epsilon^4 \ & \ 2+3\epsilon-10\epsilon^2+3\epsilon^4 \ \\[2pt]
\end{array}
$$\\
Finally, one obtains the equality
\begin{footnotesize}
$$\overline T^{\,2} \,=\, 11\left(4X^3+7X^2-6X+19\right)\!
\big(1458X^3 + 12640X^2 - 56666X + 59141 - (107X^2 + 7296X - 5508)Y\big)^2$$
\end{footnotesize}
\!\!from the Fourier $q_{\scriptscriptstyle\ast}$-expansions of \,$\overline T$, $X$ and $Y$\!.
\end{proof}

\vspace{1truemm}

\begin{rem}\label{norm}
{\rm  Two polynomials $f_1(x), f_2(x)$ satisfying (\ref{almost_equation}) could be explicitly given:
\begin{scriptsize}
$$
\begin{array}{ccl}
f_1(x) & \!\!\!\!=\!\!\! & \left(-3468 - 14558 \epsilon - 7094 \epsilon^2 + 5265 \epsilon^3 + 2866 \epsilon^4\right)x^3 \,+\,
\left(-30028 - 126111 \epsilon - 61527 \epsilon^2 + 45617 \epsilon^3 + 24848 \epsilon^4\right)x^2 \, \,+\\[7pt]
& &  \left(134152 + 563914 \epsilon + 276359 \epsilon^2 - 204056 \epsilon^3 - 111479 \epsilon^4\right)x \,+\,
139612 - 587393 \epsilon - 288873 \epsilon^2 + 212613 \epsilon^3 + 116417 \epsilon^4,\\[7pt]
f_2(x) & \!\!\!\!=\!\!\! & \left(251 + 1054 \epsilon + 522 \epsilon^2 - 382 \epsilon^3 - 210 \epsilon^4\right)x^2 \,+\,
\left(17231 + 72451 \epsilon + 35622 \epsilon^2 - 26226 \epsilon^3 - 14358 \epsilon^4\right)x \, \,+\\[7pt]
& &  -13166 - 55263 \epsilon - 26705 \epsilon^2 + 19969 \epsilon^3 + 10808 \epsilon^4.
\end{array}
$$
\end{scriptsize}
\!\!Note that the function
$$\sqrt{-11}\,\,\dfrac{\,\big(X - 11 \epsilon - 6 \epsilon^2 + 4 \epsilon^3 + 2 \epsilon^4\big)^{\!5}\,\,\widehat T\,}
{\,f_1(X) \,+\, f_2 (X)\,Y} $$
generates \,$\Q(\zeta)\!\big(X_{ns}(11)\big)$ over \,$\Q(\zeta)\!\big(X^+_{ns}(11)\big)$.
Its norm to \,$\Q\!\left(\sqrt{-11}\,\right)\!\!\big(X_{ns}(11)\big)$ is, up to sign, the generator
for \,$\Q\big(X_{ns}(11)\big)$ over \,$\Q\big(X^+_{ns}(11)\big)$ given by
\mbox{\,$\left(4X^3+7X^2-6X+19\right)^2T$\!.\, The point} in taking the trace \,$\overline T$\,
rather than the norm of the function
\,$\big(X - 11 \epsilon - 6 \epsilon^2 + 4 \epsilon^3 + 2 \epsilon^4\big)^{\!5}\,\,\widehat T\,$
as a generator is the much simpler expression for its square inside \,$\Q[X,Y]$.
}
\end{rem}

\vspace{1truemm}

\begin{rem}\label{trace1}
{\rm
The trace \,$\overline T$\!,\, hence also the generator \,$T$\!,\, can be explicitly given in terms of Siegel functions
from the data in the proof of Proposition \ref{prop2} and in the statements of Proposition~\ref{prop1} and Lemma \ref{K-L}.}
\end{rem}

\vspace{1truemm}

\begin{rem}\label{trace2}
{\rm
The proof of Proposition \ref{prop2} is in fact an alternative explicit way of deducing the equation~(\ref{xns}).
Indeed, the square of the trace \,$\overline T$\,
is defined over $\Q$, hence either \,${^\varsigma}\overline T=\overline T$\, or \,${^\varsigma}\overline T=-\overline T$\,
must hold for the complex conjugation $\varsigma$.
In the first case, \,$\overline T$ would already be defined over $\Q$ and the minimal field of definition for its values
at the cusps should necessarily be~$\Q(\zeta)$. But all the values of the function \,$11\left(4X^3+7X^2-6X+19\right)$\,
at the cusps can be computed from those given in Proposition \ref{prop1} and can be checked to be squares in~$\Q(\epsilon)$.
So~\,$\overline T/\sqrt{-11}$\, lies in $\Q\!\left(X_{ns}(11)\right)$.}
\end{rem}

\vspace{2truemm}

\section{The Jacobian of $X_{ns}(11)$}\label{Jacobian}

In this section we exploit the equations (\ref{xns+}), (\ref{xns}) and
the regular differentials on the modular curve $X_0(121)$ to make explicit
an isomorphism defined over $\Q$ between the new part of its Jacobian and the Jacobian $J_{ns}(11)$ of the
curve $X_{ns}(11)$. First, we need to recall some concepts and tools.

\vspace{2truemm}

An abelian variety $\bar{\mathcal{A}}$ is an \emph{optimal quotient}\, of an abelian variety $\mathcal{A}$ over a
number field~$K$ if there is a surjective morphism $\bar\pi\colon \mathcal{A}\!\rightarrow\!\bar{\mathcal{A}}$ defined
over $K$ with connected kernel. In such a case, $\bar{\mathcal{A}}$ is determined by the pull-back
$\bar\pi^*(\Omega^1_{\bar{\mathcal{A}}/K})$ of the regular differentials defined over~$K$.
Any basis \,$\omega_1,\cdots,\omega_n$ of this $K$-vector space yields a lattice for $\bar{\mathcal{A}}$
as a complex torus by taking in $\C^n$ the elements of the form
$$\left(\int_\gamma\!\omega_1\,, \,\dots\,,\int_\gamma\!\omega_n\right)\!,$$
for $\gamma$ running over the homology group $H_1(\mathcal{A},\Z)$.
For any surjective morphism $\pi\colon \mathcal{A}\!\rightarrow\!\mathcal{B}$ of abelian varieties  defined over $K$,
there exists a morphism $\bar\pi\colon \mathcal{A}\!\rightarrow\!\bar{\mathcal{A}}$ onto an optimal quotient along with
an isogeny $\mu\colon \bar{\mathcal{A}}\!\rightarrow\!\!\mathcal{B}$, both defined over $K$, making commutative
the following diagram:
$$\xymatrix{
 \!\mathcal{A} \ar@{->}[dr]^{\pi}\ar@{ ->}[d]_{\bar\pi}& \\
 \!\!\bar{\mathcal{A}} \ar@{->}[r]^{\!\!\!\mu}& \mathcal{B} \\
 }$$
Whenever $\mathcal{A}$ is the Jacobian of a curve $X$ defined over $K$ and $\mathcal{B}$ is an elliptic curve,
there is a morphism $\phi\colon X\!\rightarrow\!\mathcal{B}$ defined over $K$ inducing $\pi$,
hence there exists a non-constant morphism $\bar\phi\colon X\!\rightarrow\!\bar{\mathcal{A}}$\,
for which the diagram
\begin{equation}\label{optimal}
\xymatrix{
 \,X \ar@{->}[dr]^{\phi}\ar@{ ->}[d]_{\bar\phi}& \\
 \!\!\bar{\mathcal{A}} \ar@{->}[r]^{\!\!\!\mu}& \mathcal{B} \\
 }
\end{equation}
also holds.

\vspace{2truemm}

We need to consider the classical modular setting $X=\!X_0(N)$.
The \emph{new part} $J_0(N)^{\new}$ of the Jacobian $J_0(N)$ is the optimal quotient defined as follows.
Take the usual local parameter \,\mbox{$q\!:=\!e^{\,2\pi i\tau}$} for~$\tau$ in the complex upper half-plane, so that
the space $S_2(\Gamma_0(N))$ of weight $2$ cuspforms can be seen, through their Fourier $q$-expansions,
inside the $\C$-vector space \,$\C[[q]]$.
The map \,$h(q)\longmapsto h(q)\,dq/q$\, yields an isomorphism from $S_2(\Gamma_0(N))$ to the space
of regular dif\-fe\-ren\-tials \,$\Omega^1_{X_0(N)/\C}$\,
such that \,$\Omega^1_{X_0(N)/\Q}$ is the image of the $\Q$-vector space
of rational cuspforms, $S_2(\Gamma_0(N))\cap\Q[[q]]$.
The abelian variety $J_0(N)^{\new}$ corresponds to the subspace $S_2(\Gamma_0(N))^{\new}\cap \Q[[q]]$
which is orthogonal to the subspace of oldforms with respect to the Petersson product.

\vspace{2truemm}

Let us now focus on the genus $6$ curve $X_0(121)$ and its Jacobian $J_0(121)$.
There are exactly four $\Q$-isogeny classes of elliptic curves with conductor $121$.
They can be represented by the Weierstrass models
\begin{equation}\label{elliptic}
\begin{array}{lll}
E_A:\quad y^2+xy+y=x^3+x^2-30x-76,\\[5pt]
E_B:\quad y^2+y=x^3-x^2-7x+10,\\[5pt]
E_C:\quad y^2+xy=x^3+x^2-2x-7,\\[5pt]
E_D:\quad y^2+y=x^3-x^2-40x-221,
\end{array}
\end{equation}
with respective labels $A1, B1, C1, D1$ in Cremona's tables~\cite{Cre}.
These four elliptic curves are optimal quotients of $J_0(121)$ and $J_0(121)^{\new}$ over $\Q$.
Each of them admits a non-constant morphism from $X_0(121)$
defined over $\Q$ such that the
pull-back of the invariant differential is, up to sign, the regular differential on $X_0(121)$
corresponding to the normalized newform in $S_2(\Gamma_0(121))$ attached to the $\Q$-isogeny class of
the elliptic curve. Let $f_{\!A}, f_{\!B}, f_C, f_{\!D}$ denote these newforms. Thus, the sets
\begin{small}
\begin{equation*}
\!\!\!
\Lambda_A\!:=\left\{\int_\gamma\!f_A(q)\frac{dq}{q}\right\}, \quad
\Lambda_B\!:=\left\{\int_\gamma\!f_B(q)\frac{dq}{q}\right\}, \quad
\Lambda_C\!:=\left\{\int_\gamma\!f_C(q)\frac{dq}{q}\right\}, \quad
\Lambda_D\!:=\left\{\int_\gamma\!f_D(q)\frac{dq}{q}\right\},
\end{equation*}
\end{small}
\!\!\!for $\gamma$ running over the homology group $H_1(X_0(121),\Z)$, are the lattices corresponding to
the elliptic curves $E_A, E_B, E_C, E_D$, respectively.

\vspace{2truemm}

By Section 8 in \cite{Chen}, $J_0(121)^{\new}$ is isogenous over $\Q$ to the Jacobian $J_{ns}(11)$.
In order to show this explicitly, we exploit the basis for the space of regular differentials
\,$\Omega^1_{X_{ns}(11)/\Q}$ that was used in \cite{DFGS} to prove that the automorphism group
of $X_{ns}(11)$ is isomorphic to Klein's four group.
The quotients of the curve by the involutions in this group are
the elliptic curves~$E_B, E_C$ and the genus~$2$ curve given by the equation
\begin{equation}\label{genustwo}
\cX\,: \ \  y^2 \,=\, -\left(4x^3-4x^2-28x+41\right)\!\left(4x^3+7x^2-6x+19\right)\!.
\end{equation}
One has a commutative diagram
\begin{equation*}\label{diagrama}
\xymatrix{
& & X_{ns}(11) \ar@{->}[lld]_{\phi_{\!B}}
\ar@{->}[d]^{\!\phi_{\!\cX}}
\ar@{->}[rdd]^{\phi_{\!D}}
\ar@{->}[ldd]_{\phi_{\!A}}
\ar@{->}[rrd]^{\phi_{C}} & & \\
E_B & & \cX \ar@{->}[ld]^{\pi_{\!A}}
\ar@{->}[rd]_{\pi_{\!D}} &  & E_C \\
& E_A &  & E_D &
}
\end{equation*}
with degree $2$ morphisms $\phi_{\!\scriptscriptstyle B}, \phi_{\!\scriptscriptstyle\cX},
\phi_{\!\scriptscriptstyle C}$ and degree $3$ morphisms $\pi_{\!\!\scriptscriptstyle A}, \pi_{\!\scriptscriptstyle D}$
which, on the models for these curves given by (\ref{xns+}), (\ref{xns}), (\ref{elliptic}) and (\ref{genustwo}),
are defined as follows:
\begin{small}
$$
\phi_{\!\scriptscriptstyle B}(x,y,t)=\left(x,y\right), \qquad  \quad
\phi_{\!\scriptscriptstyle\cX}(x,y,t)=\left(x,\left(2y+1\right)t\right), \qquad \quad
\phi_{\scriptscriptstyle C}(x,y,t)=\left(-x-1,(t+x+1)/2\right),
$$\\[-20pt]
$$\begin{array}{l}
\!\!\pi_{\!\scriptscriptstyle A}(x,y) = \left(\!\dfrac{\,- 13 x^3 + 13 x^2 + 102 x -147  \,}{\,4x^3-4x^2-28x+41\,} \, ,
\, \dfrac{\,\left(- 8 x^3 + 19 x^2 - 10 x + 6\right)y\,}{\,2\left(4x^3-4x^2-28x+41\right)^2\,}
+ \dfrac{\,9 x^3 - 9 x^2 - 74 x + 106\,}{\,2\left(4x^3-4x^2-28x+41\right)\,}\!\right)\!,\\[25pt]
\!\!\pi_{\!\scriptscriptstyle D}(x,y) = \left(\!\dfrac{\,27 x^3 + 138 x^2 - 101 x -749   \,}{\,4x^3+7x^2-6x+19\,} \, ,
\, \dfrac{\,121\left(x^3 - x^2 - 29 x - 53\right)y\,}
{\,2\left(4x^3+7x^2-6x+19\right)^2\,} \,-\, \dfrac{\,1\,}{\,2\,}\!\right)\!.
\end{array}
$$
\end{small}

\noindent We claim that the elliptic curves $E_A, E_B, E_C, E_D$ are also optimal quotients of $J_{ns}(11)$ over~$\Q$.
Indeed, it suffices to replace the morphism $\phi$ in (\ref{optimal}) with each of 
$\phi_{\!\!\scriptscriptstyle A},\phi_{\!\scriptscriptstyle B},\phi_{\!\scriptscriptstyle C},\phi_{\!\scriptscriptstyle D}$
and take into account that neither $E_A$ nor $E_D$ admits an isogeny of degree $2$ or $3$ defined over~$\Q$
to another elliptic curve; see Table 1 in \cite{Cre}.
This fact suggests that $J_0(121)^{\new}$ and $J_{ns}(11)$ could be isomorphic.

\vspace{2truemm}

\begin{thm}\label{isom}
The abelian variety $J_0(121)^{\new}$ is isomorphic over \,$\Q$ to $J_{ns}(11)$.
Moreover, there exists an isogeny defined over \,$\Q$
$$J_{ns}(11) \longrightarrow E_A \times E_B\times E_C\times E_D$$
with kernel isomorphic to $(\Z/2\Z)^4\!\times\!(\Z/3\Z)^2$.
\end{thm}

\begin{proof}
By computational means, we produce period matrices $\mathit{\Omega}_{\new}$ and $\mathit{\Omega}_{ns}$
for the abelian varieties $J_0(121)^{\new}$ and $J_{ns}(11)$, respectively, such that
\begin{equation}\label{period_matrices}
\mathit{\Omega}_{ns}\,\,\mathcal{M} \,=\, \mathcal{N}\,\mathit{\Omega}_{\new}
\end{equation}
for some matrices \,$\mathcal{M}\!\in\!\GL_8(\Z)$ and \,$\mathcal{N}\!\in\!\GL_4(\Q)$.
This relation amounts to the first part of the statement.

\vspace{2truemm}

To make the computations easy, we choose in both cases a suitable basis of regular differentials.
Specifically, the directions of each basis are the pull-backs to $J_0(121)^{\new}$ or to~$J_{ns}(11)$, respectively,
of the spaces \,$\Omega^1_{E_A/\Q}$\,, $\Omega^1_{E_B/\Q}$\,, $\Omega^1_{E_C/\Q}$\,, $\Omega^1_{E_D/\Q}$.
Moreover, the elements in each basis yield, by integration over the homology group
of the respective abelian variety,
the above elliptic lattices $\Lambda_A, \Lambda_B, \Lambda_C, \Lambda_D$.
This is possible due to the fact that the elliptic curves $E_A, E_B, E_C, E_D$ are optimal quotients of both
abelian varieties over $\Q$. The elements of such bases are then determined up to sign, and the
lattices $\Lambda_{\new}$ and $\Lambda_{ns}$ of $J_{0}(121)^{\new}$ and $J_{ns}(11)$, respectively, defined
by these bases are both contained in the lattice
$$\Lambda_A\times \Lambda_B\times \Lambda_C\times \Lambda_D$$
attached to the abelian variety $E_A\times E_B\times E_C\times E_D$.

\vspace{2truemm}

The obvious basis of regular differentials to be taken for $J_0(121)^{\new}$ is the one
corresponding to the newforms $f_{\!A}, f_{\!B}, f_C, f_{\!D}$.
A basis for $H_1(X_0(121),\Z)$ can be explicitly given as a set of modular symbols. The following
subset provides a basis for the new part:
$$\begin{array}{llll}
\gamma_1=\{0,3/52\}, \quad &\gamma_2=\{ 0, 4/97\}, \quad &\gamma_3=\{ 0,19/92\}, \quad &\gamma_4=\{ 0,59/119\},\\[4pt]
\gamma_5=\{ 0,8/57\}, \quad &\gamma_6=\{ 0, 11/74\}, \quad &\gamma_7=\{ 0, 3/28\}, \quad &\gamma_8=\{ 0, 11/37\}.
\end{array}
$$
One can then compute the matrix period attached to these data:
\begin{footnotesize}
$$
\mathit{\Omega}_{\new} \,:=\,\left( \begin{array}{cccc}
\ds\int_{\gamma_1}\!\!\!f_{\!A} (q)\frac{dq}{q} & \cdots & \cdots & \ds\int_{\gamma_8}\!\!\!f_{\!A}(q)\frac{dq}{q}\\[4 pt]
\vdots &\vdots &  \vdots & \vdots\\[4 pt]
\vdots &\vdots &  \vdots & \vdots\\[4 pt]
\ds\int_{\gamma_1}\!\!\!f_{\!D} (q)\frac{dq}{q} & \cdots & \cdots & \ds\int_{\gamma_8}\!\!\!f_{\!D}(q)\frac{dq}{q}
\end{array} \right).
$$
\end{footnotesize}
\!\!Its columns regarded in $\C^4$ are a $\Z$-basis for $\Lambda_{\new}$. A $\Z$-basis for each of the
elliptic lattices $\Lambda_A, \Lambda_B, \Lambda_C, \Lambda_D$ can also be computed, and the quotient
group \,$\left(\Lambda_A\times\Lambda_B\times\Lambda_C\times\Lambda_D\right)\!/\Lambda_{\new}$\, is then checked to
be isomorphic to $(\Z/2\Z)^2\!\times\!(\Z/6\Z)^2$, which implies the second part of the statement.

\vspace{2truemm}

As for the basis of regular differentials on $X_{ns}(11)$, we multiply by suitable constants the invariant
differentials of the elliptic quotients and then take the corresponding pull-backs:\\[-20pt]

\begin{small}
$$
\omega_{\!\scriptscriptstyle A}:=\phi_{\!\!\scriptscriptstyle A}^*\left(\frac{4\,dx}{\,2y+x+1\,}\right), \quad
\omega_{\!\scriptscriptstyle B}:=\phi_{\!\scriptscriptstyle B}^*\left(\frac{4\,dx}{\,2y+1\,}\right), \quad
\omega_{\scriptscriptstyle C}:=\phi_{\!\scriptscriptstyle C}^*\left(\frac{4\,dx}{\,2y+x\,}\right),\,\quad
\omega_{\!\scriptscriptstyle D}:=\phi_{\!\scriptscriptstyle D}^*\left(\frac{-4\,dx}{\,2y+1\,}\right),
$$
\end{small}

\!\!
\noindent that is to say,\\
$$
\omega_{\!\scriptscriptstyle A} := \frac{\,44\,dX\,}{\,\left(2Y+1\right)T\,}\,, \qquad \omega_{\!\scriptscriptstyle B}:= \frac{\,4\,dX\,}{\,2Y+1\,}\,,\qquad
\omega_{\scriptscriptstyle C} := \frac{\,-4\,dX\,}{T}\,, \qquad \omega_{\!\scriptscriptstyle D}:=\frac{\,4\left(3 X-1\right)dX\,}{\,\left(2Y+1\right)T\,}\,,
$$\\[-5pt]
where $X,Y,T$ are the functions in Propositions \ref{prop1} and \ref{prop2}.
In order to compute a period matrix, we need to express these differentials in terms of a
plane (singular) model for the curve, which can be derived from the equations (\ref{xns+}) and (\ref{xns}).
Specifically, we take \,$T$\, and \,$Z:=\left(2Y+1\right)T$\, as generators of the function field over $\Q$,
so that
$$X \,=\, \dfrac{\,12\,T^4 \,+\, 187\,T^2 \,+\, Z^2\,}{\,4\left(T^4\,-\,11\,T^2\,+\,Z^2\right)\,}\,.$$\\[-5pt]
Thus, from the relation \,\,$TdT\!=\!-(6X^2+7X-3)dX$\, one obtains
$$\omega_{\scriptscriptstyle C}\,=\,\dfrac{\,32\left(T^4-11T^2+Z^2\right)^2\,}
{\,\left(18\,T^4+121\,T^2+7Z^2\right)\!\left(32\,T^4 + 605\,T^2 - Z^2\right)\,}\,dT$$\\[-15pt]
and
$$\omega_{\!\scriptscriptstyle A}\,=\,\dfrac{\,-11\,T\,}{\,Z\,}\,\omega_{\scriptscriptstyle C}\,, \qquad  \qquad
\omega_{\!\scriptscriptstyle B}\,=\,\dfrac{\,-T^2\,}{\,Z\,}\,\omega_{\scriptscriptstyle C}\,, \qquad \qquad
\omega_{\!\scriptscriptstyle D}\,=\,\dfrac{\,-T\left(32\,T^4 + 605\,T^2 - Z^2\right)\,}{\,4\,Z\left(T^4-11\,T^2+Z^2\right)\,}\,\omega_{\scriptscriptstyle C}\,.
$$\\[-5pt]
One can then produce, using algebraic computational software such as \emph{Maple}, a period matrix~$\mathit{\Omega}_{ns}$
attached to this basis of regular differentials. The proof concludes by checking that the matrix
$$
\mathcal{M}:=\left (\!\!\begin{array}{c}  \operatorname{Re\,}(\mathit{\Omega}_{ns})\\[2pt]
\operatorname{Im\,} (\mathit{\Omega}_{ns})\end{array}\!\!\right)^{\!\!-1}
\!\left (\!\!\begin{array}{c} \operatorname{Re\,} (\mathit{\Omega}_{\new})\\[2pt]
\operatorname{Im\,} (\mathit{\Omega}_{\new}) \end{array}\!\!\right)
$$\\[-5pt]
lies in $\!\GL_8(\Z)$ and, thus, satisfies (\ref{period_matrices}) for $\mathcal{N}$ equal to the identity matrix.
\end{proof}

\vspace{1truemm}

\begin{rem}{\rm
Though it is not needed for the proof, it can also be checked from the period matrix $\mathit{\Omega}_{ns}$ that the
quotient $\left(\Lambda_A\times\Lambda_B\times\Lambda_C\times\Lambda_D\right)\!/\Lambda_{ns}$ is isomorphic to $(\Z/2\Z)^2\!\times\!(\Z/6\Z)^2$.
Furthermore, the lattice $\Lambda_\cX$ of the Jacobian $\Jac(\cX)$, defined by the regular differentials
whose pull-backs by $\phi_{\!\scriptscriptstyle\cX}$ are \,$\omega_{\!\scriptscriptstyle A}, \omega_{\!\scriptscriptstyle D}$,
is a sublattice of $\Lambda_A\!\times\!\Lambda_D$ with quotient isomorphic to $(\Z/3\Z)^2$.
Thus, there is an isogeny defined over $\Q$
$$J_{ns}(11) \longrightarrow \Jac(\cX)\times E_B\times E_C$$
with kernel isomorphic to $(\Z/2\Z)^4$.
}
\end{rem}

\vspace{1truemm}

\begin{rem}{\rm
The normalized weight $2$ cuspforms corresponding to the pull-backs
of the spaces \,$\Omega^1_{E_A/\Q}$\,, $\Omega^1_{E_B/\Q}$\,, $\Omega^1_{E_C/\Q}$\,, $\Omega^1_{E_D/\Q}$\,
by the morphisms $\phi_{\!\!\scriptscriptstyle A},\phi_{\!\scriptscriptstyle B},\phi_{\!\scriptscriptstyle C},\phi_{\!\scriptscriptstyle D}$, res\-pec\-tive\-ly, are canonically defined by
the regular differentials \,$\omega_{\!\scriptscriptstyle A}, \omega_{\!\scriptscriptstyle B}, \omega_{\scriptscriptstyle C}, \omega_{\!\scriptscriptstyle D}$\, of $X_{ns}(11)$ in the proof
of Theorem \ref{isom}. The first coefficients of their Fourier $q_{\scriptscriptstyle\ast}$-expansions
can be computed from the expressions given in Propositions \ref{prop1} and \ref{prop2} for the generators $X,Y,T$ of the function field of the curve:

\begin{small}
$$\,\,\,\begin{array}{c||c|c}
q_{\scriptscriptstyle\ast}^n &
\dfrac{\,\sqrt{-11}\,}{\,4\left(\epsilon-2\right)^2}\,(3 - \epsilon^2)(1 + \epsilon - \epsilon^2)
\,\omega_{\!\scriptscriptstyle A}\,\dfrac{\,q_{\scriptscriptstyle\ast}\,}{\,dq_{\scriptscriptstyle\ast}\,} &
\dfrac{\,\sqrt{-11}\,}{\,4\left(\epsilon-2\right)^4}\,
(11 + \epsilon - 22 \epsilon^2 + 2 \epsilon^3 + 5 \epsilon^4)\,
\omega_{\scriptscriptstyle C}\,\dfrac{\,q_{\scriptscriptstyle\ast}\,}{\,dq_{\scriptscriptstyle\ast}\,}
\\[10pt]
\hline\hline
& & \\[-10pt]
\,q_{\scriptscriptstyle\ast} & 1 & 1 \\[4pt]
q_{\scriptscriptstyle\ast}^2 &
 -8 \epsilon - 6 \epsilon^2 + 3 \epsilon^3 +
 2 \epsilon^4  &
2 \epsilon - \epsilon^3\\[4pt]
q_{\scriptscriptstyle\ast}^3 &
-2\left(-2 - 8 \epsilon - 5 \epsilon^2 + 3 \epsilon^3 + 2 \epsilon^4\right)  &
2\left(-2 + 2 \epsilon +  \epsilon^2 -  \epsilon^3\right) \\[4pt]
q_{\scriptscriptstyle\ast}^4 &-6 \epsilon -
 3 \epsilon^2 + 2 \epsilon^3 + \epsilon^4&
2 -
 3 \epsilon^2 + \epsilon^4\\[4pt]
q_{\scriptscriptstyle\ast}^5 &
 -3 - 6 \epsilon - 2 \epsilon^2 +
 2 \epsilon^3 + \epsilon^4  &
3 - 4 \epsilon^2 + \epsilon^4 \\[4pt]
q_{\scriptscriptstyle\ast}^6 &
2\left(-3 - 6 \epsilon - 2 \epsilon^2 + 2 \epsilon^3 + \epsilon^4\right)  &
2\left(-3 + 4 \epsilon^2 -  \epsilon^4\right)\\[2pt]
\end{array}
$$

$$\!\!\!\!\!\!\!\!\!\!\!\!\!\!\!\!\!\begin{array}{c||c|c}
q_{\scriptscriptstyle\ast}^n &
\ \dfrac{\,-11\,}{\,4\left(\epsilon-2\right)\,}\,(2 - \epsilon - 3 \epsilon^2 + \epsilon^4)\,
\omega_{\!\scriptscriptstyle B}\,\dfrac{\,q_{\scriptscriptstyle\ast}\,}{\,dq_{\scriptscriptstyle\ast}\,} \ &
\dfrac{\,\sqrt{-11}\,}{\,\,4\,}\,
\epsilon\,(1 - \epsilon)(1 + \epsilon) (2 + \epsilon)\,
\omega_{\!\scriptscriptstyle D}\,\dfrac{\,q_{\scriptscriptstyle\ast}\,}{\,dq_{\scriptscriptstyle\ast}\,}
\\[10pt]
\hline\hline
& & \\[-10pt]
\,q_{\scriptscriptstyle\ast} & 1 & 1 \\[4pt]
q_{\scriptscriptstyle\ast}^2 &
0 &
2\left(-4 \epsilon + 3 \epsilon^3 +  \epsilon^4\right)\\[4pt]
q_{\scriptscriptstyle\ast}^3 &
-3 + 3 \epsilon + 4 \epsilon^2 - \epsilon^3 - \epsilon^4 &1 + \epsilon - 2 \epsilon^2 -
 3 \epsilon^3 - \epsilon^4\\[4pt]
q_{\scriptscriptstyle\ast}^4 &
2\left(2 - 4 \epsilon^2 +  \epsilon^4\right) &
2\left(-2 \epsilon^3 -  \epsilon^4\right )\\[4pt]
q_{\scriptscriptstyle\ast}^5 &
-3\left(-2 + \epsilon + \epsilon^2\right) &3 \epsilon + \epsilon^2 - 4 \epsilon^3 - 2 \epsilon^4
\\[4pt]
q_{\scriptscriptstyle\ast}^6 &
0  & 2\left(3 \epsilon +  \epsilon^2 - 4 \epsilon^3 - 2 \epsilon^4\right)
\\[2pt]
\end{array}
$$
\end{small}

\noindent Note that the cuspforms provided by the pull-backs of the invariant differentials of
the elliptic curves $E_A, E_B, E_C, E_D$ are not normalized: the leading coefficient in each case is,
up to $4$ times a unit in $\Z[\zeta]$, a power of a generator of the prime ideal over $11$.
The corresponding normalized cuspforms, as it happens for the functions $X$, $Y$ and $T/\sqrt{-11}$,
have Fourier coefficients in~$\Q(\epsilon)$.
It may be that this property holds for all primes $p\geq 11$, namely that there exists a basis of~\,$\Omega^1_{X_{ns}(p)/\Q}$
whose corresponding weight~$2$ normalized cuspforms have Fourier coefficients
in the maximal real subfield of the $p$-th cyclotomic extension of $\Q$.
}
\end{rem}

\vspace{3truemm}

It follows from Theorem \ref{isom} that $J_{ns}(11)$ is an optimal quotient of $J_0(121)$.
Nevertheless, the Riemman-Hurwitz formula implies that $X_{ns}(11)$ does not admit a non-constant
morphism from $X_0(121)$. Let us finish by giving a less obvious result: neither does its genus $2$ quotient~$\cX$.
This is related to the \emph{modularity} of curves over $\Qbar$.

\vspace{1truemm}

A curve $\cC$ defined over a number field $L$\, is called \emph{modular over $L$}\,
if there exists a non-constant morphism \,$\pi\colon\!X(N)\!\longrightarrow\cC$ defined over $L$ for some integer $N$\!.
In that case, the curve~$\cC$ is also covered by the modular curve $X_1(N^2)$, though not necessarily over $L$.
See, for instance, the appendix in \cite{Ma91}.
Both $X_{ns}(11)$ and $\cX$ admit then a non-constant morphism from~$X_1(121)$.
Let us also show that any such a morphism cannot be defined over the quadratic
field $K=\Q(\sqrt{-11}\,)$.

\vspace{1truemm}

\begin{prop}\label{points}
The curves $X_{ns}(11)$ and $\cX$ have no points defined over $K$\!.
\end{prop}
\begin{proof}
It suffices to prove the statement for $\cX$.
The origin $O$ of~$E_A$ is the only rational point on this elliptic curve; see \cite{Cre}.
Take $P$ in $E_A(K)$. Then $P+{}^{\varsigma\!}P$ lies in $E_A(\Q)$, where \,$\varsigma$ stands for the complex conjugation.
So \,${}^{\varsigma\!}P=-P$\!,\, which means that $P$ defines a rational point on the twist of $E_A$ by $K$\!.
On this twist, which has conductor $121$ and label $C2$ in \cite{Cre},
the origin is also the only rational point. Hence $P$ must be $O$.
Since the preimages of $O$ under the map~$\pi_A$ are not defined over $K$\!,\, the result follows.
\end{proof}

\vspace{1truemm}

\begin{prop}\label{endos}
All endomorphisms of $J_0(121)$ are defined over $K$\!.
\end{prop}
\begin{proof}
The Jacobian $J_0(121)$ is isogenous over $\Q$ to the product of elliptic curves
$$E_A\times E_B\times E_C\times E_D\times J_0(11)^2.$$
The elliptic curve $E_B$ is the only one in this product having complex multiplication (by the ring of integers of $K$).
As noted in the proof of Proposition \ref{points}, $E_A$ is isogenous to $E_C$ over~$K$\!.
The same holds for $E_D$ and $J_0(11)$. Specifically, the twist of $E_D$ by $K$ is the elliptic curve with conductor $11$
and label $A3$ in \cite{Cre}.
Since none of the curves in the $\Q$-isogeny class of~$E_A$ is isomorphic
to a curve in the $\Q$-isogeny class of $E_D$,
these two elliptic curves cannot be isogenous. Hence the splitting of $J_0(121)$ over $\overline{\Q}$ is
$$E_A^{\,2}\times E_B\times E_D^{\,3}$$
and it is defined over $K$\!.
The endomorphisms of $E_B$ are defined over $K$\!, so the result follows.
\end{proof}

\vspace{1truemm}

\begin{cor}\label{milibro}
Any curve $X$ of genus \,$g>1$ defined over $K$ which can be covered by $X_0(121)$ 
admits a non-constant morphism \,$X_0(121)\!\longrightarrow\!X$ defined over $K$\!.
\end{cor}

\begin{proof}
For a non-constant morphism \,$\pi: X_0(121)\!\longrightarrow\!X$ defined over $\Qbar$,
Proposition \ref{endos} implies that 
$$\pi^*(\Omega^1_{X/\overline{\Q}}) \cap \Omega^1_{X_0(121)/K}$$ 
is a $K$-vector space of dimension $g$. The results in Section 2 of \cite{BGGP} show then that there exist functions 
in $K(X_0(121))$ which generate the field $K(X)$ over $K$\!.
\end{proof}

\vspace{1truemm}

\begin{cor}
The curves $X_{ns}(11)$ and $\cX$ cannot be covered by $X_1(M)$ over $K$\!, for any~$M$\!,
nor by $X_0(121)$ over \,$\overline{\Q}$.
\end{cor}

\begin{proof}
Since the cusp at infinity on both $X_1(M)$ and $X_0(121)$ is defined over $\Q$, the result follows from
Proposition \ref{points} and Corollary \ref{milibro}.
\end{proof}

\vspace{15truemm}

\bibliography{X_ns(11)}{}
\bibliographystyle{alpha}

\end{document}